\documentclass{amsart}      
\usepackage{amssymb,amsmath,mathrsfs,eufrak,enumerate,latexsym,enumerate,color}
\usepackage{amsthm}
\usepackage[english]{babel}                   
\newcommand{\Stab}{\operatorname{Stab}}

\newcommand{\nor}{\trianglelefteq}         
\newcommand{\abs}[1]{\left\vert#1\right\vert}         
\newcommand{\set}[1]{\left\{#1\right\}}               
\newcommand{\seq}[1]{\left<#1\right>}                 

\newcommand{\PGL}{{\operatorname{PGL}}}
\newcommand{\PGammaL}{{\operatorname{P\Gamma L}}}

\newtheorem{thm}{Theorem}
\newtheorem{lem}{Lemma}
\newtheorem{pro}{Proposition}
\newtheorem{defn}{Definition}
\begin{document}
\title{On the primary 
coverings of finite solvable and symmetric groups}
\date{\today}
\author{Francesco Fumagalli}
\thanks{Francesco Fumagalli. Dipartimento di Matematica e Informatica ``Ulisse Dini'',
         viale Morgagni 67/A, 50134 Firenze, Italy. francesco.fumagalli@unifi.it. Member of INDAM.}

\author{Martino Garonzi}
\thanks{Martino Garonzi. Universidade de Bras\'{i}lia, 
Campus Universit\'{a}rio Darcy Ribeiro, 
Departamento de Matem\`{a}tica
Bras\'{i}lia - DF
70910-900, Brasil. mgaronzi@gmail.com.
Supported by Funda\c{c}\~{a}o de Apoio \`a Pesquisa do Distrito Federal (FAPDF) - demanda espont\^{a}nea 03/2016, and by 
Conselho Nacional de Desenvolvimento Cient\'ifico e Tecnol\'ogico (CNPq) - Grant numbers 302134/2018-2, 422202/2018-5.} 


\begin{abstract}
A primary covering of a finite group $G$ is a family of proper 
subgroups of $G$ whose union contains the set of elements of $G$ 
having order a prime power. We denote with $\sigma_0(G)$ 
the smallest size of a primary covering of $G$, and call it the primary
covering number of $G$. We study this number and compare it with 
its analogous $\sigma(G)$, the covering number, for the classes of 
groups $G$ that are solvable and symmetric.  
\end{abstract}

\maketitle

\section{Introduction}\label{sec:intro}
A \emph{covering} of a finite group $G$ is a family of 
proper subgroups of $G$ whose union equals $G$. The 
\emph{covering number} of $G$ is defined as the minimal size 
of a covering of $G$, and it is denoted by $\sigma(G)$ (by Cohn \cite{Cohn}). 
A group admits a covering unless it is cyclic, in which case 
it is generally set $\sigma(G)$ to be equal to $\infty$ 
(with the convention that $n<\infty$ for every integer $n$). 
The covering number has been studied by many authors and 
in particular $\sigma(G)$ was determined when $G$ is solvable 
by Tomkinson \cite[Theorem 2.2]{Tom} and when $G$ is symmetric 
by Mar\'oti \cite[Theorem]{Maroti2005} (see also \cite{Swartz}, 
\cite{Kappe} and \cite{Opp}).

Given a subset $\Pi$ of $G$ we may be interested in the minimal 
number of proper subgroups of $G$ whose union contains $\Pi$. 
In this paper we focus on the set of primary elements, complementing 
the work done in \cite{Fuma}. A \emph{primary element} of $G$ is an 
element of $G$ whose order is some prime power. 
We define $G_0$ to be the set of primary elements of $G$ and a  
\emph{primary covering} of $G$ to be a family 
of proper subgroups of $G$ whose union contains $G_0$. We set 
$\sigma_0(G)$ the smallest size of a primary covering of $G$.
and call it the \emph{primary covering number of $G$}. 
Observe that $G$ admits primary coverings if and only if 
$G$ is not a cyclic $p$-group for no prime $p$, so in this case 
we define $\sigma_0(G) = \infty$, with the convention that 
$n < \infty$ for every integer $n$. Clearly, we always have 
$\sigma_0(G) \leq \sigma(G)$. Moreover, a deep result 
(\cite[Theorem 1]{Kantor}) shows that a primary covering 
of any finite group is never a unique conjugacy class of a 
proper subgroup.

In this paper we study $\sigma_0(G)$ when $G$ is solvable and 
when $G$ is a symmetric group $S_n$.

Our main result about solvable groups is the following.

\begin{thm}\label{thm:solvable}
Let $G$ be a finite solvable group which is not a cyclic $p$-group, for every prime $p$. If $G/G'$ is not a $p$-group, then $\sigma_0(G)=2$. Otherwise,
$\sigma_0(G)=\sigma(G)$.
\end{thm}

Our results on the primary covering number for $S_n$ can be summarized 
in the following statement.
\begin{thm}\label{thm: main S_n}
The following hold for $n\geq 3$.
\begin{enumerate}
\item $\sigma_0(S_3)=4$ and $\sigma_0(S_6)=7$ 
(see Lemmas \ref{lem:S_5} and \ref{lem:S_6}).
\item If $n=2^a$ for some $a>1$, then $\sigma_0(S_n)= 1 +  
\frac{1}{2}{n\choose n/2}$ (see Proposition \ref{pro:S_n n=2power}).
\item If $n\neq 3^{\epsilon}2^a$, for $\epsilon\in\set{0,1}$ 
and $a>1$, then $\sigma_0(S_n)= 1+ {n\choose n_2}$, where $n_2$ denotes 
the maximum power of $2$ that divides $n$ (see Proposition \ref{pro:S_n, n_not_32^a} 
and Lemma \ref{lem:S_10}). 
\item If $n=3\cdot 2^a$, with $a\geq 2$, then 
$c_1\leq \sigma_0(S_n)\leq c_2,$
where 
\begin{align*}
c_1 &=\begin{cases} 117 & \textrm{ if } a=2,\\
                 1+{n-1\choose 2^{a}-1} 
                 & \textrm{ if } a\geq 3,
      \end{cases}\\ 
c_2& = 2+{n-1\choose 2^{a}-1}+
\sum_{i=2}^{2^{a+1}}{n-i\choose 2^{a-1}-1}
\end{align*}
(see Proposition \ref{pro:main 32^a}).
\end{enumerate}
\end{thm}
As a comparison with the work done in \cite{Fuma}, observe the following. Denoting by $\gamma_0(G)$ the normal primary covering number of a finite group $G$, that is the smallest natural number such that
$$G_0\subseteq \bigcup_{i=1}^{\gamma_0(G)}\bigcup_{g\in G} H_i^g$$
for some proper - pairwise non-conjugate - subgroups $H_i$ of $G$, 
by Propositions 1 and 5 in \cite{Fuma} we have that if $G$ is a group whose order is not a prime power and $G$ is either solvable or symmetric then $\gamma_0(G)=2$.

%
%
%
%
\section{Solvable groups}\label{sec:solvable}
In this section we determine the primary covering number for every finite
solvable group $G$. For the basic results, as well as the notation, concerning  
solvable groups we refer the reader to \cite[Chapter 5]{Robinson}.

We start with the following trivial observations that hold in a general 
context.
\begin{lem}\label{lem:solvable_0}
Assume that $G$ is a finite group.
\begin{enumerate}
\item If $G/G'$ is not a $p$-group, for some prime $p$, 
then $\sigma_0(G)=2$.
\item If $N\nor G$, then $\sigma_0(G) \leq \sigma_0(G/N)$. Moreover, 
if $N$ is contained in the Frattini subgroup of $G$, then 
$\sigma_0(G)=\sigma_0(G/N)$. 
\item If $M$ is a maximal subgroup of $G$ such that 
$\sigma_0(M) > \sigma_0(G)$, then $M$ belongs to every minimal primary 
covering of $G$.
\end{enumerate}
\end{lem}
\begin{proof}
(1) We trivially have that $G_0\subseteq H\cup K$, where $H$ and $K$ are 
maximal subgroups containing $G'$ of coprime indices.

(2) Since any primary covering of $G/N$ lifts to a primary covering of $G$, we have $\sigma_0(G) \leq \sigma_0(G/N)$. Also, any 
subgroup in a primary covering of $G$ can be replaced by a maximal subgroup 
containing it. Thus when $N\leq \Phi(G)$ we have $\sigma_0(G)=\sigma_0(G/N)$.  

(3) Let $M$ be a maximal subgroup of $G$ such that $\sigma_0(G) < \sigma_0(M)$ 
and let $H_1,\ldots,H_n$ be any primary covering of $G$ of size 
$n=\sigma_0(G)$. Of course, the family $\{H_i \cap M\}_{i=1}^n$ 
covers $M_0$ (the set of primary elements of $M$), so 
since $\sigma_0(M) > n$ we deduce that there exists $i \in \{1,\ldots,n\}$ 
such that $H_i \cap M = M$, in other words $M \leq H_i$. Being $M$ a maximal 
subgroup of $G$ we deduce $M=H_i$.
\end{proof}

\begin{lem}\label{lem:solvable_formula}
Let $N$ be a complemented minimal normal subgroup of a solvable group $G$ 
and let $b$ be the number of complements of $N$ in $G$. If $b>1$ then 
$b \geq |N|$.
\end{lem}

\begin{proof}
Let $E=\mathrm{End}_G(N)$. By a result of Gasch\"{u}tz, \cite[Satz 3]{Gasch}, 
we know that $b=|N|^{\varepsilon} |E|^{\beta-1}$ where $\varepsilon$ is 
$0$ or $1$ according to whether $N$ is central or not in $G$, and $\beta$ 
is the number of non-Frattini chief factors $G$-isomorphic to $N$, in a 
chief series of $G$ starting with $N$. It follows that $b \geq |N|$ if 
$N$ is not central, and if $N$ is central then $|N|=p$ is a prime number 
and $E$ is the field with $p$ elements, so that $\beta \neq 1$ 
(being $b>1$ by hypothesis) hence $b \geq |E| = p = |N|$.
\end{proof}

Recall Tomkinson's result (\cite[Theorem 2.2]{Tom}) which states that if $G$ is a finite solvable group then $\sigma(G)=q+1$ where $q$ is the order of the smallest chief factor of $G$ with more than one complement.

\ 

\noindent
{\it Proof of Theorem \ref{thm:solvable}.} 
Let $G$ be a non-cyclic finite solvable group. If $\abs{G/G'}$ is not a prime 
power, then by Lemma \ref{lem:solvable_0} (1) we have $\sigma_0(G)=2$. Thus 
assume that $G/G'$ is a $p$-group for some prime $p$ and define $\alpha$ to 
be the smallest order of a chief factor of $G$ 
admitting more than one complement (this is well-defined because $G$ is not 
a cyclic $p$-group). We need to show that $\sigma_0(G)=1+\alpha$. 
By the aforementioned result of Tomkinson, 
$\sigma_0(G) \leq \sigma(G)=1+\alpha$. Let $N$ 
be a normal subgroup of $G$ such that $\sigma_0(G)=\sigma_0(G/N)$ with 
$|G/N|$ minimal with this property. Let $K/N$ be a minimal normal subgroup 
of $G/N$, then $\sigma_0(G/N) < \sigma_0(G/K)$, hence by Lemma 
\ref{lem:solvable_0} (2) $K/N$ admits a complement $M/N$ in $G/N$ which is a 
maximal subgroup. Being 
$\sigma_0(G/N) < \sigma_0(G/K) = \sigma_0(M/N)$, we deduce that $M/N$ 
appears in every minimal primary covering of $G/N$, hence all $b$ 
complements of $K/N$ in $G/N$ appear in a fixed minimal primary covering 
of $G/N$. However, no element of $K/N$ belongs to any complement of 
$K/N$, hence $\sigma_0(G/N) \geq 1+b$. If $b \neq 1$, then $b \geq |K/N|$ by 
Lemma \ref{lem:solvable_formula} and this implies the result. 
Assume now $b=1$. Then $G/N$ is a direct product $K/N \times M/N$ 
hence $K/N$, being a central chief factor, is cyclic of prime order 
and, being an epimorphic image of the $p$-group $G/G'$, we deduce that 
$|K/N|=p$. 
Moreover $M/N$ is nontrivial, because $G/N \not \cong C_p$ (being 
$\sigma_0(G)=\sigma_0(G/N)$), so $M/N$ has nontrivial abelianization. 
Since $G/G'$ is a $p$-group it follows that 
$M/N$ projects onto $C_p$. This $C_p$ is therefore a chief factor above 
$K/N$ which is $G$-isomorphic to $K/N$. In particular, the number of 
non-Frattini chief factors $G$-isomorphic to $C_p$ is $\beta\geq 2$. 
This contradicts the formula in the proof of Lemma 
\ref{lem:solvable_formula}.\qed

%
%
%
%
\section{Symmetric groups}\label{sec:sym_alt}
We introduce some notation that will be frequently used. 
Let $n$ be a positive integer and set $\Omega := \set{1,2,\ldots,n}$ on 
which the symmetric group $S_n$ acts naturally. Given natural numbers $k\geq 1$ and
$x_1\geq x_2\geq \ldots\geq x_k\geq 1$ such that $\sum_{i=1}^k x_i=n$, we 
denote with $(x_1,x_2,\ldots,x_k)$ both the partition of $n$ whose 
parts are precisely the $x_i$ and the set of all permutations in $S_n$ 
having as cyclic type this partition. In particular, $(n)$ denotes the 
set of $n$-cycles of $S_n$.  

The maximal subgroups of $S_n$ split in three different classes, according 
to their action on $\Omega$: intransitive, imprimitive and 
primitive subgroups (see \cite{DM}). 

Any intransitive maximal subgroup is the setwise 
stabilizer of a set of cardinality $m$, for some $1\leq m<n/2$. 
In particular, any such subgroup is conjugate to the stabilizer 
of the set $\set{1,2, \ldots, m}\subseteq \Omega$, which we denote with 
$X_m$. It is therefore isomorphic to $S_m\times S_{n-m}$ and its index
is ${n\choose m}$. We set 
$$\mathcal{X}_m=\set{\textrm{conjugates of } X_m\simeq S_m\times S_{n-m}}$$
and  
$$\mathcal{X}=\bigcup_{1\leq m<n/2}\mathcal{X}_m=
\set{\textrm{intransitive maximal subgroups of } S_n}.$$

The imprimitive maximal subgroups of $S_n$ are the stabilizers of 
partitions of $\Omega$ into equal-sized subsets.
If $d$ is any proper nontrivial divisor of $n$ we set $W_d$ the 
stabilizer of the partition
$$\set{\set{1,2,\ldots,d}, \set{d+1,d+2,\ldots,2d},\ldots, 
\set{n-d+1,n-d+2,\ldots,n}}.$$
Note that $W_d$ is isomorphic to the wreath product $S_d\wr S_{n/d}$, and it has index $\frac{n!}{(d!)^{n/d}\cdot(n/d)!}$. Also note that any imprimitive 
maximal subgroup of $S_n$ is conjugate to $W_d$, for some proper 
nontrivial divisor $d$ of $n$. We set 
$$\mathcal{W}_d=\set{\textrm{conjugates of } W_d\simeq S_d\wr S_{n/d}}$$
and  
$$\mathcal{W}=\bigcup_{1< d\vert n,\, d\neq n}\mathcal{W}_d=
\set{\textrm{imprimitive maximal subgroups of } S_n}.$$

Finally set 
$$\mathcal{P}=\set{\textrm{proper primitive maximal subgroups of } S_n}$$
where proper means that both $S_n$ and $A_n$ are not members 
of $\mathcal{P}$.

\begin{lem} \label{swap}
If $a>b$ are positive integers then $$a!^b b! \geq b!^a a!$$ with equality if and only if $b=1$.
\end{lem}

\begin{proof}
If $b=1$ we have equality, now assume $b \geq 2$, so that $a \geq 3$. Since the stated inequality is equivalent to $$\frac{\ln(a!)}{a-1} \geq \frac{\ln(b!)}{b-1}$$ it is enough to prove that the function $\ln(a!)/(a-1)$ is increasing, hence we may assume $b=a-1$. So we need to prove that $a!^{a-2} \geq (a-1)!^{a-1}$ which is equivalent to $a^{a-1} \geq a!$, which is actually a strict inequality being $a \geq 3$.
\end{proof}

The following lemmas are part of \cite[Corollary 1.2 and Lemma 2.1]{Maroti2002}.

\begin{lem}[Mar\'oti Lemma 2.1, On the orders of primitive groups] \label{lemab}
Let $m>1$ be an integer and suppose $m=a_1b_1=a_2b_2$ with $a_1,a_2,b_1,b_2$ positive integers at least $2$, $b_1 \geq a_1$, $b_2 \geq a_2$, $a_1 \leq a_2$ and (consequently) $b_1 \geq b_2$. Then $$b_1!^{a_1} a_1! \geq b_2!^{a_2} a_2!$$with equality if and only if $a_1=a_2$ and $b_1=b_2$. Moreover if $p$ is the smallest prime divisor of $m$ and $d$ is any divisor of $m$ with $1 < d < m$ then $$(m/p)!^p p! \geq (m/d)!^d d!$$with equality if and only if $d=p$.
\end{lem}

\begin{proof}
Observe that since $b_2 \geq a_2 \geq 2$ we have $b_2^{b_2} \geq b_2! b_2 \geq b_2! a_2$.
\begin{align*}
b_1!^{a_1} a_1! & \geq b_2!^{a_1} (b_2+1)^{a_1} \cdots b_1^{a_1} a_1! \\ & \geq b_2!^{a_1} a_1! b_2^{a_1(b_1-b_2)} \\ & \geq b_2!^{a_1} a_1! (b_2! a_2)^{(a_1/b_2)(b_1-b_2)} \\ & = b_2!^{a_1} a_1! (b_2!a_2)^{(a_2-a_1)} \\ & \geq b_2!^{a_1} a_1! b_2!^{a_2-a_1} (a_1+1) \cdots a_2 \\ & = b_2!^{a_2} a_2!.
\end{align*}
If equality holds then all of the above inequalities are equalities and it is 
easy to deduce that $a_1=a_2$ and consequently $b_1=b_2$. To deduce the last 
statement, observe that it is trivial if $m=p$ so assume this is not the case, 
so that $p^2 \leq m$. Choose $a_1=p$, $b_1=m/p$, then $a_1 \leq b_1$ and the 
inequality $a_1 \leq a_2$ will be true for every choice of a divisor $a_2>1$ of 
$m$, by minimality of $p$. If $d^2 \leq m$ then choose $a_2=d$, proving the 
strict inequality with equality if and only if $d=p$. If $d^2 > m$ then 
$d > m/d$ hence Lemma \ref{swap} implies that $d!^{m/d}(m/d)! \geq (m/d)!^d d!$, so it is enough to show that $(m/p)!^p p! \geq d!^{m/d}(m/d)!$, which follows from the above choosing $a_2=m/d$, and again by Lemma \ref{swap} equality does not hold in this case being $m/d \neq 1$.
\end{proof}
 
The orders of the different primitive and imprimitive maximal 
subgroups of $S_n$ can be compared as in the following lemma. 
\begin{lem}\label{lem:order_max_impr}
Let $n\geq 2$, then the following hold.
\begin{enumerate}
\item For every proper nontrivial divisor $d$ of $n$, we have 
$$\abs{W_{n/p}} = (n/p)!^p p! \geq (n/d)!^d d! = \abs{W_d},$$ where $p$ is the smallest prime 
divisor of $n$, and equality holds if and only if $d=p$;
\item if $P\in\mathcal{P}$, then $\abs{P}<3^n$, and when $n>24$, then 
$\abs{P}<2^n$.
\end{enumerate}
In particular, for $n\geq 12$  every $M\in\mathcal{W}\cup\mathcal{P}$ has order 
$$\abs{M}\leq 2\big(\lfloor n/2\rfloor\big)!\big(n-\lfloor n/2\rfloor\big)! $$
with equality if and only if $M\in \mathcal{W}_{n/2}$.
\end{lem}
\begin{proof}
(1) follows from Lemma \ref{lemab}, and (2) is \cite[Corollary 1.2]{Maroti2002}.

We now prove the last part of the Lemma.\\ 
Let first $M\in\mathcal{P}$. By (2) the order of $M$ is bounded above by 
$3^n$. 
Recall Stirling's bound $k!> e(k/e)^k$, which holds for every $k \geq 2$ and can be proved by noting that
$$\frac{k^k}{k!} = \frac{k^{k-1}}{(k-1)!} = \prod_{i=1}^{k-1} \left( \frac{i+1}{i} \right)^i < \prod_{i=1}^{k-1} (e^{1/i})^i = e^{k-1}.$$
We deduce that 
$$2\big(\lfloor n/2\rfloor\big)!\big(n-\lfloor n/2\rfloor\big)!
> g(n)=\begin{cases} 2e^2 \left(\frac{n}{2e}\right)^n & \textrm{ if } n \textrm{ is even,}\\
 e (n+1) \left( \frac{n-1}{2e} \right)^{n-1} & \textrm{ if } n \textrm{ is odd.}
\end{cases}$$
Computation shows that for every $n\geq 14$ the function $g(n)\geq 3^n$. The cases
$n=12$ and $n=13$ can be done by inspection and the proof for the primitive case
is completed.
%

Assume now that $M\in\mathcal{W}$, in particular $m$ is not a prime number. 
Then by (1) we have $\abs{M}\leq\abs{W_{n/p}}$,
where $p$ is the smallest prime number that divides $n$. 
The result is therefore trivial if $n$ is even. Let $n$ be odd. 
We need to prove that 
\begin{equation}\label{eq:order_max_2}
R=\frac{(n+1)(((n-1)/2)!)^2}{(n/p)!^p p!}> 1.
\end{equation}
Observe that, being $p^2 \leq n$, the map $x \mapsto x^{n-1}/n^x$ is an increasing function in the interval $1 \leq x \leq p$, so that $p^{n-1}/n^p \geq 3^{n-1}/n^3$. Using the inequalities $e(n/e)^n \leq n! \leq en(n/e)^n$ we see that

\begin{align*}
R & \geq \frac{e^2 (n+1) \left( \frac{n-1}{2e} \right)^{n-1}}{epn^p\left(\frac{n}{ep}\right)^n} 
= 2e^2\cdot \frac{n+1}{n-1}\cdot  \frac{p^{n-1}}{n^p}\cdot  \left( \frac{n-1}{2n} \right)^n  \\ 
& > 2e^2 \cdot \frac{3^{n-1}}{n^3} \cdot \left( \frac{n-1}{2n} \right)^n = \frac{2e^2}{3n^3} \left( \frac{3(n-1)}{2n} \right)^n \geq 1
\end{align*}
whenever  $n\geq 22$. The case $12 \leq n \leq 21$ can be done by inspection.
\end{proof}
\noindent
In the sequel we will need the following result.

\begin{lem}\label{lem:S_n_impr}
Let $n$ be even and $W\in\mathcal{W}_{n/2}$. 
Assume that $n=\sum_{i=1}^k2^{a_i}$ is a partition of $n$ with 
$a_1\geq a_2\geq \ldots\geq a_k\geq 1$ that does not contain
subpartitions of $n/2$, and let $\Pi$ be the conjugacy class
of elements of $S_n$ of type $(2^{a_1},2^{a_2},\ldots, 2^{a_k})$. 
Then 
$$|W \cap \Pi| = \frac{|\Pi|}{|S_n:W|} \cdot 2^{k-1}.$$

In particular, when $n$ is a power of $2$ and $\Pi=(n)$, the set 
of $n$-cycles, then 
$$\abs{W\cap \Pi}=\frac{\abs{W}}{n}= (n/2)!(n/2-1)!$$
\end{lem}
\begin{proof} 
Double counting the size of the set 
$$\{(x,W)\ \vert \ x \in \Pi \cap W,\ W \in \mathcal{W}_{n/2}\},$$ 
we find that 
$$|W \cap \Pi| = \frac{|\Pi|}{|G:W|} \cdot r,$$ 
where $r$ is the number of elements of $\mathcal{W}_{n/2}$ containing a fixed 
element of $\Pi$. By the assumption 
that the partition defining $\Pi$
does not contain partitions of $n/2$ as subpartitions,
the elements of $\Pi \cap W$ move the two imprimitivity blocks of $W$, 
hence $r = 2^{k-1}$, since no cycle can fix either block, and, once we split up the elements in the $2^{a_1}$-cycle, we have two choices for each 
$2^{a_j}$-cycle, $2\leq j\leq k$.
\end{proof}

%
%
Moreover, we will make use of the following notation and terminology introduced in \cite{Maroti2005}. 

\begin{defn}\label{def:def_unbeat}
Let $\Pi$ be a set of permutations of $S_n$.
We will say that a collection $\mathcal{H}=\set{H_1,\ldots, H_m}$ 
of $m$ proper subgroups of $S_n$ is \emph{definitely unbeatable} 
on $\Pi$ if the following three conditions hold: 
\begin{enumerate}
\item $\Pi\subseteq \bigcup_{i=1}^m H_i$,
\item $\Pi\cap H_i\cap H_j=\varnothing$ for every $i\neq j$,
\item $\abs{M\cap \Pi}\leq \abs{H_i\cap \Pi}$ for every $1\leq i\leq m$ 
and every proper subgroup $M$ of $S_n$ not belonging to $\mathcal{H}$. 
\end{enumerate}
\end{defn}
If $\mathcal{H}$ is definitely unbeatable on $\Pi$, then
$\abs{\mathcal{H}}=\sigma(\Pi)$, where $\sigma(\Pi)$ denotes 
the least integer $m$ such that $\Pi$ is a subset of the 
union of $m$ proper subgroups of $S_n$. Moreover, 
we also say that $\mathcal{H}$ is \emph{strongly definitely unbeatable}
on $\Pi$ if the three conditions above hold and the third one always 
holds with strict inequalites. Note that in the case when $\mathcal{H}$ 
consists of maximal subgroups and it is strongly definitely unbeatable on $\Pi$, 
then $\mathcal{H}$ is the unique minimal cover of the elements of $\Pi$ that uses 
only maximal subgroups (see also \cite[Lemma 3.1]{Swartz}).\\

We start our considerations on the primary covering number of $S_n$ 
by producing a general upper bound. Here and in the
following if $p$ is any prime, we define $n_p$ to be the $p$-part of $n$, that is 
the maximum power of $p$ that divides $n$. 
\begin{lem}\label{lem:S_n_trivial}$\null$
\begin{enumerate} 
\item If $n$ is a power of $2$, then 
$\sigma_0(S_n)\leq 1+\frac{1}{2}{n\choose n/2}$,
\item if $n$ is not a power of $2$, then $\sigma_0(S_n)\leq 1+{n\choose n_2}$.
\end{enumerate}
\end{lem}
\begin{proof}
Note that in any case the alternating groups $A_n$ contains every permutation
of odd order, therefore in order to exhibit a primary covering for $S_n$, we 
may add to $\set{A_n}$ those subgroups that contain $2$-elements
(that are odd permutations). \\
{\it{(1)}} Assume $n=2^a$ with $a\geq 2$. Then every $2$-element of $S_n$ 
stabilizes a $2$-block partition of $\Omega$ and therefore 
every $2$-element belongs to an imprimitive maximal subgroup 
of type $\mathcal{W}_{n/2}$.
Since the number of such partitions (and subgroups) is 
$\frac{1}{2}{n\choose n/2}$ the lemma is proved in this case. \\
{\it{(2)}} The number of subsets of order $n_2$ of
the set $\Omega=\set{1,2,\ldots, n}$ is ${n\choose n_2}$ which is an odd 
number (see for instance \cite[Lemma 1.8]{Isaacs}), and it is easy to prove by induction that every 
$2$-element of $S_n$ belongs to some stabilizer 
$S_{\Delta}$ with $\Delta$ a subset of $\Omega$ of cardinality $n_2$, 
in other words if we write $n$ as a sum of distinct powers of $2$ then there 
exists a subsum that equals $n_2$. 
Since the number of these stabilizers is exactly ${n\choose n_2}$ also this 
point is proved. 
\end{proof}

We already have enough ingredients to complete the proof in
the case $n=2^a$, with $a\geq 2$. 
\begin{pro}\label{pro:S_n n=2power}
If $n=2^a \geq 4$ then $\sigma_0(S_n)=1+\frac{1}{2}{n \choose n/2}$ and 
a minimal primary covering is given by $\set{A_n}\cup\mathcal{W}_{n/2}$.
\end{pro}
\begin{proof} 
A direct inspection shows that $\sigma_0(S_4)=4=
1+\frac{1}{2}{4 \choose 2}$. Thus in the following we assume $a>2$. \\ 
By Lemma \ref{lem:S_n_trivial} we know that 
$\sigma_0(S_n)\leq 1+\frac{1}{2}{n \choose n/2}$.\\
Assume that $H$ is a maximal subgroup of $S_n$. 
Then either 
$H\cap (n)=\varnothing$, or $H$ is an imprimitive subgroup, 
or by \cite[Theorem 3]{Jones} the subgroup $H$ 
satisfies $\PGL_d(q)\leq H
\leq \PGammaL_d(q)$, where $n=(q^d-1)/(q-1)$ for some $d\geq 2$. Note that 
in this last case we necessarily have $d=2$ and $q=2^a-1$ a Mersenne prime.
To see this observe that from the equality $2^a=(q^d-1)/(q-1)$ one easily deduces that $d=2$ so that $2^a=q+1$, now $q$ cannot be a square since $2^a-1 \equiv 3 \mod 4$, and if $q=p^m$ is an odd power of the prime $p$, the usual factorization of $x^m+1=(x+1)(x^{m-1}-x^{m-2}+\ldots +1)$ implies that $q$ must be a prime.
By Lemma \ref{lem:S_n_impr}  and the fact the elements of order $n$ in 
$\mathrm{PGL}(2,q)$ are in number of $2^{a-2}(2^a-1)(2^a-2)$ 
(use \cite[II, Satz 7.3]{Hup}), we deduce in any case that
$$\abs{H\cap (n)}\leq (n/2)!(n/2-1)!$$
with equality if and only if $H\in\mathcal{W}_{n/2}$.
This shows that the set $\mathcal{W}_{n/2}$ is strongly definitely 
unbeatable on $\Pi=(n)$, 
and therefore we obtain that 
$$\sigma_0(S_{n})\geq \frac{1}{2}{n\choose n/2}.$$
To complete this case, assume that 
$\sigma_0(S_{n})= \frac{1}{2}{n\choose n/2}$. 
Then, by the above, the collection $\mathcal{W}_{n/2}$ must be the unique
minimal primary covering for $S_n$.  
By Bertrand's postulate there is a prime number $p$ between 
$n/2$ and $n$; we reach a contradiction by noting that 
$p$-cycles do not belong to imprimitive subgroups of type $\mathcal{W}_{n/2}$.
\end{proof}

We assume now that $n\not\in\set{3^{\epsilon}\cdot 2^a\vert\, \epsilon\in\set{0,1},\, a\geq 0}$.\\
We deal separately with the case $n=5$.

\begin{lem}\label{lem:S_5}
For $n=5$ we have that $\sigma_0(S_5)=1+{5\choose 1}=6$ and
$\set{A_5}\cup \mathcal{X}_1$ is the unique minimal primary covering of $S_5$.
\end{lem}
\begin{proof}
We already know that $\set{A_5}\cup \mathcal{X}_1$ is a primary covering 
for $S_5$ and therefore $\sigma_0(S_5)\leq 6$. Assume by contradiction that 
$\mathcal{C}$ is a primary covering of smaller cardinality. Inside $S_5$ there 
are six subgroups of order $5$, therefore, as $\abs{\mathcal{C}}\leq 5$, there 
exists one element of $\mathcal{C}$ containing at least two different Sylow
$5$-subgroups. 
But the only proper subgroup of $S_5$ containing more than one subgroup of 
order $5$ is $A_5$. Thus $A_5\in \mathcal{C}$ and the remaining members of 
$\mathcal{C}$ 
cover all of the odd $2$-elements of $S_5$, which are thirty $4$-cycles and ten $2$-cycles. 
Any maximal subgroup isomorphic to $S_4$ contains 
precisely six $4$-cycles and six $2$-cycles, any $S_3\times S_2$ contains 
no $4$-cycles and four $2$-cycles and any Frobenius group $5:4$ contains 
ten $4$-cycles and no $2$-cycles.
Therefore, if we assume that $\mathcal{C}$ contains respectively $a_1$ 
subgroups in 
$\mathcal{X}_1$ (that is isomorphic to $S_4$),
$a_2$ subgroups in $\mathcal{X}_2$ (that is isomorphic to $S_3\times S_2$), and 
$a_3$ primitive subgroups isomorphic to $5:4$, we obtain the following system 
of Diophantine inequalities
\begin{equation*}\label{eq_S_5__2}
\begin{cases} 
a_1+a_2+a_3\leq 4 & \\
3a_1+5a_3\geq 15 & \\
6a_1+4a_2\geq 10. & 
\end{cases}
\end{equation*} 
The only integer solution of this system is $(a_1,a_2,a_3)=(2,0,2)$, 
but then if
$\Stab_{S_5}(i)$ and $\Stab_{S_5}(j)$ are the two elements of 
$\mathcal{X}_1$ in $\mathcal{C}$
we have that the permutation $(ij)$ is not covered by elements of 
$\mathcal{C}$, 
which is a contradiction. 
\end{proof}

Let $n\geq 7$ and $n\not\in\set{3^{\epsilon}\cdot 2^a\vert\, \epsilon\in
\set{0,1},\, a\geq 0}$ and write the $2$-adic expansion of $n$ as 
$$n=2^{a_1}+2^{a_2}+\ldots +2^{a_t},$$ 
where $a_1>a_2>\ldots >a_t\geq 0$ and $t\geq 2$. Note that $n_2=2^{a_t}$
and also that when $t=2$ then $a_1\geq a_2+2$.\\ 
We define $\Pi$ to be the following conjugacy class of permutations
$$\Pi=\begin{cases} (2^{a_1},2^{a_2},\ldots,2^{a_t}) 
               & \textrm{if } n\not\equiv t \,\, (\textrm{mod } 2),\\
                     (2^{a_1-1},2^{a_1-1},2^{a_2},\ldots, 2^{a_t}) 
               & \textrm{if } n\equiv t \,\, (\textrm{mod } 2).
\end{cases} $$

The set $\Pi$ consists of odd permutations, that is 
$A_n\cap \Pi=\varnothing$. \\

The computation of $\sigma_0(S_n)$ in this case depends on the following 
proposition.
\begin{pro}\label{pro:unbeatable}
Assume that $n\not\in\set{3{^{\epsilon}}\cdot 2^a\vert\, 
\epsilon\in\set{0,1},\, a\geq 1}$. If $n$ is odd and $n\geq 15$ 
or if $n$ is even, $n\geq 22$ and $n\neq 40$, the collection of subgroups 
$\mathcal{X}_{n_2}$ is strongly definitely unbeatable on $\Pi$. 
\end{pro}

For the proof of Proposition \ref{pro:unbeatable} we need 
the following number-theoretic result.
\begin{lem}\label{lem:n_2 vs n/2}
Using the above assumptions and notation, define
%
$$s=\begin{cases} \sum_{i=1}^t a_i +a_1-1 &\textrm{ if } n\equiv t 
                          \,\, (\textrm{mod } 2)\\
                  \sum_{i=1}^t a_i &\textrm{ if } n\not\equiv t \,\, 
                  (\textrm{mod } 2)\end{cases}$$
and 
$$f(n)=2^{s+1}\frac{{n\choose n_2}}{{n\choose \lfloor n/2\rfloor}}.$$
Then $$f(n)< 1$$
holds if and only if either $n$ is odd and $n\geq 15$ or $n$ is even, 
$n\geq 22$ and $n\neq 40$. 
\end{lem}
\begin{proof}
Since $a_i\leq a_1-(i-1)$ for every $i=1,2,\ldots, t$, we have
$$s\leq \sum_{i=1}^t a_i +a_1-1\leq ta_1+a_t-1-\frac{(t-2)(t-1)}{2}.$$
The coefficient $a_t$ satisfies $a_t \leq a_1+1-t$ and therefore 
\begin{equation*}
s\leq\begin{cases} ta_1  &\textrm{ if } n \textrm{ is odd, }\\
                   (t+1)a_1-2 &\textrm{ if } n \textrm{ is even,}
                   \end{cases}
\end{equation*}
where we used the fact that $(t^2-3t+2)/2\geq 0$ and
$(t^2-t+2)/2\geq 2$  for every $t\geq 2$. 
Set $l=\log_2(n) \geq a_1$. Being $t\leq a_1-a_t+1$, 
and according to our assumptions $n\geq 15$ if $n$ is odd
and $n\geq 22$ if $n$ is even, 
we have that $t\leq l$ 
if $n$ is odd and $t\leq l-1$ if $n$ is even. 
Thus we obtain
\begin{equation}\label{eq:n/2_vs_n_2.1}
s\leq\begin{cases} l^2 &\textrm{ if } n \textrm{ is odd, }\\
                   l^2-2 &\textrm{ if } n \textrm{ is even.}
                   \end{cases}
\end{equation} 
By considering the binomial expansion of  
$(1+1)^n$ we have that 
${n\choose \lfloor n/2\rfloor}\geq \frac{2^n}{n+1}$, that is, 
being $l\leq \log_2(n+1)\leq l+1$,
\begin{equation}\label{eq:n/2_vs_n_2.2}
{n\choose \lfloor n/2\rfloor}\geq 2^{n-(l+1)}.
\end{equation}
We distinguish now the different cases.\\

\noindent
Case $n$ odd.\\
Then ${n\choose n_2}=n=2^l$ and (\ref{eq:n/2_vs_n_2.1}) and 
(\ref{eq:n/2_vs_n_2.2}) imply that
$$f(n)\leq 2^{l^2+2l+2-n}.$$
We have that $l^2+2l+2-n< 0$ for every $n\geq 43$. \\
The cases $15\leq n\leq 41$ can be checked by direct computation. 
(Note that $f(n)>1$ for $5\leq n\leq 13$).\\

\noindent
Case $n$ even.\\
Assume first that $t\geq 4$.\\
Then $n \geq 2^{a_t}(1+2+\ldots+2^{t-1}) = n_2(2^t-1) \geq 15 n_2$, hence  
$n_2\leq \lfloor n/15\rfloor$ and 
$${n\choose n_2}\leq {n\choose \lfloor n/15\rfloor}.$$
Note that if $k$ is any non-negative integer then $k^k/k! < e^k$ by the well-known series expansion of the exponential function, therefore if $k \leq n$ then
$${n\choose k} = \frac{n}{k} \cdot \frac{n-1}{k} \cdot \ldots \cdot \frac{n-k+2}{k} \cdot \frac{n-k+1}{k} \cdot \frac{k^k}{k!} < (ne/k)^k.$$
Using the fact that 
for every real number $x$ we have $x-1\leq \lfloor x\rfloor \leq x$, we deduce that 
\begin{equation}\label{eq:n/2_vs_n_2.3}
{n\choose n_2}\leq \left(\frac{15en}{n-15}\right)^{n/15}=
2^{\frac{n}{15}\log_2\left(\frac{15en}{n-15}\right)}.
\end{equation}
Combining (\ref{eq:n/2_vs_n_2.1}), (\ref{eq:n/2_vs_n_2.2}) and 
(\ref{eq:n/2_vs_n_2.3}) we have that
$$f(n)\leq 2^{l^2+l-n\left(1-\frac{1}{15}\log_2\left(\frac{15en}
{n-15}\right)\right)}.$$
Note that $f(n)< 1$ if  $l^2+l-n\left(1-\frac{1}{15}
\log_2\left(\frac{15en}{n-15}\right)\right)< 0$, which 
is true for every $n\geq 72$. 
The cases $22\leq n\leq 70$, with $t\geq 4$, can be checked by a 
direct computation. \\

Let now $t=2$ or $t=3$.\\
In both cases we have that the value of $s$ is bounded from above by 
$3l-3$, because when $t=2$ then $a_1\geq a_2+2$ and so
$$s=2a_1+a_2-1\leq 3a_1-3\leq 3l-3$$
and when $t=3$ then
$$s=a_1+a_2+a_3\leq a_1+a_1-1+a_1-2\leq 3l-3.$$
Moreover, when $t=2$, since $a_1\geq a_2+2$, we have that 
$n_2\leq \lfloor n/5\rfloor$. Therefore, by considerations 
analogous to the ones above, we have that $f(n)< 1$ if
\begin{equation}\label{eq:eq:n/2_vs_n_2.5}
4l-1-n\left(1-\frac{1}{7}\log_2\left(\frac{7en}{n-7}\right)\right)<
0 \,\, \textrm{ when } t=3,
\end{equation}
\begin{equation}\label{eq:eq:n/2_vs_n_2.6}
4l-1-n\left(1-\frac{1}{5}\log_2\left(\frac{5en}{n-5}\right)\right)<
0 \,\, \textrm{ when } t=2.
\end{equation}
Now (\ref{eq:eq:n/2_vs_n_2.5}) holds for every $n\geq 62$, while 
(\ref{eq:eq:n/2_vs_n_2.6})
holds for every $n\geq 114$.\\ 
As above, the intermediate cases can be checked by computation,
the only exceptions being all $n$ even with $n\leq 20$ and $n=40$. 
\end{proof}

\noindent 
{\it{Proof of Proposition \ref{pro:unbeatable}.}}
To prove that $\mathcal{X}_{n_2}$ is strongly definitely 
unbeatable on $\Pi$ we need to show that the three conditions 
of the Definition \ref{def:def_unbeat} are satisfied.
Conditions (1) and (2) are straightforward (note that when 
$t=2$ we have that $a_1-1>a_2$). We show condition (3), that is, 
for every $X_{n_2}\in\mathcal{X}_{n_2}$ and 
every maximal subgroup $M$ of $S_n$, which is not the  
stabilizer of a $n_2$-subset, the proportion $$c(M):=\frac{\abs{M\cap \Pi}}{\abs{X_{n_2}\cap \Pi}}=\frac{\abs{M}}{\abs{X_{n_2}}}m_M< 1,$$
where $m_M$ denotes the number of conjugates of $M$ containing a 
fixed element of $\Pi$. The above expression for $c(M)$ was obtained via a double counting, as in the proof of Proposition \ref{pro:S_n n=2power}, applied to compute both $|M \cap \Pi|$ and $|X_{n_2} \cap \Pi|$, keeping in mind that a given element of $\Pi$ belongs to exactly one member of $\mathcal{X}_{n_2}$.

Assume first that $M$ is an intransitive maximal subgroup not 
in $\mathcal{X}_{n_2}$.\\ 
If $M\cap \Pi\neq \varnothing$ we necessarily 
have that $M$ is the stabilizer of a union of disjoint $2$-power sized subsets, 
say $A_i$ for some $i\geq 1$. Since $\abs{M\cap \Pi}\leq 
\abs{\Stab(A_1)\cap \Pi}$, we can assume that $M$ coincides with 
the stabilizer of a single subset of cardinality some power of $2$, 
say $2^c$, with $n_2< 2^c<n/2$, being $2^{a_1} \geq n/2$, and therefore $c\in\set{a_2,\ldots,a_{t-1}}$ 
when $n\not\equiv t$ (mod $2$) and $c\in\set{a_1-1,a_2,\ldots,a_{t-1}}$ 
when $n\equiv t$ (mod $2$). 
Now, except in the case $n\equiv t$ (mod $2$) and $c=a_1-1$, 
we have that $m_M=1$ and 
$$c(M)=\frac{{n\choose n_2}}{{n\choose 2^{c}}}<1,$$
since $n_2=2^{a_t}<2^c$.
Otherwise $m_M=2$ and, since $a_t\leq a_1-2$, we obtain
$$c(M)=2\cdot \frac{{n\choose 2^{a_t}}}{{n\choose 2^{a_1-1}}}
       \leq 2\cdot \frac{{n\choose 2^{a_1-2}}}{{n\choose 2^{a_1-1}}}<1.$$
 
Assume now that $M$ is a primitive or an imprimitive maximal subgroup
of $S_n$.
Then by Lemma \ref{lem:order_max_impr}, 
$$\abs{M\cap \Pi}\leq\abs{M}\leq  
\frac{2n!}{{n\choose \lfloor n/2\rfloor}},$$ 
while
$$\abs{X_{n_2}\cap \Pi}=\frac{n_2!(n-n_2)!}{2^s},$$
where $s=\sum_{i=1}^t a_i$ 
if $n\not\equiv t$ (mod $2$), and $s=\sum_{i=1}^t a_i+a_1-1$ 
if $n\equiv t$ (mod $2$).
Therefore 
$$c(M)=\frac{\abs{M\cap \Pi}}{\abs{X_{n_2}\cap \Pi}}
\leq \frac{\abs{M}}{\abs{X_{n_2}\cap \Pi}}\leq 
2^{s+1}\cdot\frac{{n\choose n_2}}{{n\choose \lfloor n/2\rfloor}}$$
Lemma \ref{lem:n_2 vs n/2} proves that $c(M)<1$, whenever 
$n$ is odd and $n\geq 15$, or $n$ is even and $n\geq 22$ and $n\neq 40$. \hfill $\square$\\

We treat now the case when $n\leq 20$ or $n=40$. 
\begin{pro}\label{pro:unbeatable_small}
For $n\in\set{5,7,9,10,11,13,14,18,20,40}$, the collection 
of subgroups $\mathcal{X}_{n_2}$ is strongly definitely unbeatable on $\Pi$
if and only if $n\not\in\set{5, 10}$. 
\end{pro}
\begin{proof}
As in the proof of Proposition \ref{pro:unbeatable}, the conditions (1) and (2) 
of Definition \ref{def:def_unbeat} are straightforward to prove, so we limit ourselves to show that condition (3) holds.

In the sequel, denote by $M$ a maximal subgroup of $S_n$ 
not in $\mathcal{X}_{n_2}\cup\set{A_n}$.

If $M$ is intransitive the inequality $|M \cap \Pi| \leq |X_{n_2} \cap \Pi|$ 
holds in all cases, with equality if and only if
$M\in\mathcal{X}_{n_2}$ (see the proof of Proposition \ref{pro:unbeatable}). So we may concentrate on the imprimitive and 
primitive maximal subgroups. We treat the various cases separately.

Case $n=5$. 
Then $\Pi=(4,1)$ and if we take $M$ to be a 
primitive subgroup isomorphic to the Frobenius group $5:4$, we have that 
$$\abs{M\cap \Pi}=10>6=\abs{X_{n_2}\cap \Pi}.$$ 
Therefore $\mathcal{X}_{n_2}$ is not unbeatable on $\Pi$.

Case $n=7$.  
Then $\Pi=(2,2,2,1)$ and $|X_{n_2} \cap \Pi|=15$. There is only one class of 
transitive (primitive) maximal subgroups of $G=S_7$ (not containing $A_7$).
They have order $42$ and their intersection with $\Pi$ has size $7$.

Case $n=9$. 
Then $\Pi=(8,1)$ and $|X_{n_2} \cap \Pi|=7!$, which is larger than $432$, 
the maximal size of a primitive subgroup of $S_9$ not containing $A_9$.
Since subgroups in $\mathcal{W}_3$ have trivial intersection with $\Pi$, 
we have that $\mathcal{X}_{n_2}$ is unbeatable on $\Pi$.

Case $n=10$. Then $\Pi=(4,4,2)$ and if $M\in \mathcal{W}_{5}$ we have, by Lemma \ref{lem:S_n_impr},
$$\abs{M\cap \Pi}=1.800>1.260=\abs{X_{n_2}\cap \Pi}, $$
which shows that $\mathcal{X}_{n_2}$ is not unbeatable on $\Pi$.

Case $n=11$. 
We have that $\Pi=(4,4,2,1)$ and $|X_{n_2} \cap \Pi| = 56.700,$ which is larger 
than $110$, the maximal size of a primitive subgroup of $S_{11}$ not containing $A_{11}$.

Case $n=13$. 
Then $\Pi=(4,4,4,1)$ and $|X_{n_2} \cap \Pi| \geq 1,24\cdot 10^{6}$.
This number is larger than the maximal size of a primitive subgroup 
of $S_{13}$ not containing $A_{13}$, which is $156$. 

Case $n=14$.
Then $\Pi=(8,4,2)$ and $|X_{n_2} \cap \Pi| \geq 1,4 \cdot 10^{7}$.
If $M$ is a primitive maximal subgroup of $S_{14}$ not containing $A_{14}$ then 
$|M \cap \Pi| < |M| < 3^{14} < |X_{n_2} \cap \Pi|$.
Assume that $M$ is imprimitive. If $M\in\mathcal{W}_2$, then 
$$|M \cap \Pi| \leq |M| = 2^7 \cdot 7! = 645.120 < |X_{n_2} \cap \Pi|.$$
If $M\in\mathcal{W}_7$ then by Lemma \ref{lem:S_n_impr},
$$|M \cap \Pi| = \frac{|\Pi|}{|G:M|} \cdot 2^2 = 3.175.200 
< |X_{n_2} \cap \Pi|.$$  

Case $n=18$. 
Then $\Pi=(8,8,2)$ and $|X_{n_2} \cap \Pi| = \frac{16!}{2^7} \geq 1,63\cdot 10^{11}$. 
If $M$ is a primitive maximal subgroup of $S_{18}$ not containing $A_{18}$, then 
$$|M \cap \Pi| < |M| < 3^{18} = 387.420.489 < |X_{n_2} \cap \Pi|.$$
If $M$ is an imprimitive maximal subgroup of $S_{18}$, not in 
$\mathcal{W}_9$, then 
$$|M \cap \Pi| < |M|\leq (6!)^3 \cdot 3! 
      \leq 2,24\cdot 10^{9}< |X_{n_2} \cap \Pi|.$$
If $M\in\mathcal{W}_9$ then by Lemma \ref{lem:S_n_impr},
$$|M \cap \Pi| = \frac{|\Pi|}{|G:M|} 2^2 \leq 4,12 \cdot 10^{9}
                              < |X_{n_2} \cap \Pi|.$$ 

Case $n=20$. 
Then $\Pi=(8,8,4)$ and $|X_{n_2} \cap \Pi| = \frac{16!}{2^7} \cdot 3!
\geq 9,8\cdot 10^{11}$. \\
If $M$ is a primitive maximal subgroup of $S_{20}$ not containing $A_{20}$, then 
$$|M \cap \Pi| \leq 3^{20} \leq 3,49\cdot 10^{9} < |X_{n_2} \cap \Pi|.$$ 
If $M$ is an imprimitive maximal subgroup of $G$, not in $\mathcal{W}_{10}$, 
then 
$$|M \cap \Pi| \leq |M| \leq |S_5 \wr S_4| = (5!)^4 4! 
\leq 4,98\cdot 10^{9}< |X_{n_2} \cap \Pi|.$$
If $M\in\mathcal{W}_{10}$ then by Lemma \ref{lem:S_n_impr}, 
$$|M \cap \Pi| = \frac{|\Pi|}{|G:M|} 2^2 
\leq 2,06\cdot 10^{11}  < |X_{n_2} \cap \Pi|.$$ 

Case $n=40$. 
Then $\Pi=(16,16,8)$ and $|X_{n_2} \cap \Pi| = \frac{32!}{2^9} \cdot 7!\geq 2,59\cdot 10^{36}$. 
If $M$ is a primitive maximal subgroup of $S_{40}$ not containing $A_{40}$ then 
$|M \cap \Pi|<|M|\leq 2^{40}\approx  1,10 \cdot 10^{12}< |X_{n_2} \cap \Pi|$. If $M$ is an imprimitive maximal subgroup of $G$, not in $\mathcal{W}_{20}$, then 
$$|M \cap \Pi| < |M| \leq |S_{10} \wr S_4| = 10!^4 \cdot 4! \leq 4,17\cdot 10^{27}< |X_{n_2} \cap \Pi|.$$ 
If $M\in\mathcal{W}_{20}$ then by Lemma \ref{lem:S_n_impr},
$$|M \cap \Pi| = \frac{|\Pi|}{|G:M|} 2^2 \leq 1,16 \cdot 10^{34}
< |X_{n_2} \cap \Pi|.$$
The proof is now complete.
\end{proof}

\begin{lem}\label{lem:S_10}
For $n=10$ we have that $\sigma_0(S_{10})= 46$.
\end{lem}
\begin{proof}
We already know that $\set{A_{10}}\cup \mathcal{X}_2$ is a primary covering 
and therefore $\sigma_0(S_{10})\leq 46$. 
Assume by contradiction that $\mathcal{C}$ is a primary covering of smaller cardinality. $A_{10}$ belongs to $\mathcal{C}$ because the maximal intersection of a maximal subgroup of $S_{10}$ distinct from $A_{10}$ with the set of the $72.576$ elements of cycle structure $(5,5)$ is $576$, realized by the class of subgroups $\mathcal{W}_5$, and $72.576/576 = 126 > 46 \geq \sigma_0(S_{10})$. 
To conclude it is enough to show that at least $45$ maximal subgroups are needed to cover the elements of cycle type $(4,4,2)$. We were able to do this using the programs
\cite{gap,gurobi}. More specifically, this works as follows. Let $\mathscr{M}$ be the set of maximal subgroups of $G=S_{10}$ and let $C$ be the conjugacy class of elements of cycle structure $(4,4,2)$ in $G$. Let $r_M$ be a variable for every $M \in \mathscr{M}$ and, for every $x \in C$, define $\mathscr{M}_x := \{M \in \mathscr{M}\ :\ x \in M\}$. We found that
$$\min \Big\{\sum_{M \in \mathscr{M}} r_M\ \Big\vert\ r_M \in \{0,1\}\ 
\forall M \in \mathscr{M},\ \sum_{M \in \mathscr{M}_x} r_M \geq 1\ \forall x \in C\Big\} = 45.$$
For this, \cite{gap} was used to compute the sets $\mathscr{M}_x$ and \cite{gurobi} was used to solve the optimization problem.
\end{proof}

We can now complete this case.
\begin{pro}\label{pro:S_n, n_not_32^a}
Assume that $n\neq 10$ and $n\neq 3^{\epsilon}2^a$, 
for every $\epsilon\in\set{0,1}$ and every $a\geq 0$. Then $\sigma_0(S_n)= 1+ {n\choose n_2}$.
\end{pro} 
\begin{proof}

\noindent
The case $n=5$ has been done in Lemma \ref{lem:S_5}.\\
By Propositions \ref{pro:unbeatable}
and \ref{pro:unbeatable_small} we have that 
$$\sigma_0(S_n)\geq \sigma(\Pi)=\abs{\mathcal{X}_{n_2}}={n\choose n_2}.$$ 
Moreover, if it were $\sigma_0(S_n)={n\choose n_2}$ then, by 
the strongly definitely unbeatable property, $\mathcal{X}_{n_2}$ would 
be a primary covering for $S_n$, which is impossible by \cite[Theorem 1]{Kantor}, or 
simply because this collection does not cover the primary elements 
acting fixed-point-freely. Therefore, $\sigma_0(S_n)> {n\choose n_2}$ and then
Lemma \ref{lem:S_n_trivial} completes the proof.
Note that in this situation a minimal primary covering is given by 
$\set{A_n}\cup\mathcal{X}_{n_2}$.
\end{proof}

Finally, assume now that $n=3\cdot 2^a$ for some $a\geq 0$.
The case $n=3$ is trivial, thus assume that $a\geq 1$. 
We first deal with the case $n=6$. 
\begin{lem}\label{lem:S_6}
The primary covering number for $S_6$ is $7$. A minimal primary covering 
is (a conjugate to) the following
$$\mathcal{C}=\set{A_6, X_1,X_1^{(12)},X_1^{(13)}, P_1, 
P_1^{(34)},P_1^{(35)}},$$
where $X_1=\Stab_{S_6}(\set{1})\in\mathcal{X}_1$ and 
$P_1=\seq{(3465),(123)(456)}$ belongs to the family $\mathcal{P}$ of primitive maximal subgroups isomorphic to $S_5$.
\end{lem}
\begin{proof}
A direct check with GAP shows that $\mathcal{C}$ is a covering for the set of 
primary elements of $S_6$.

Assume by contradiction that $\mathcal{D}$ is a primary covering (consisting of 
maximal subgroups of $G$ and) containing less than seven elements. 

We first show that $A_6\in\mathcal{D}$. If this is not the case, then the 
class $\Pi_0$ of $5$-cycles should be covered by at most six maximal subgroups, 
which are either $1$-point stabilizers, that is elements of $\mathcal{X}_1$, or 
primitive maximal subgroups, that is elements of $\mathcal{P}$, and in both cases 
they are all isomorphic to $S_5$. 
Note that $\abs{\Pi_0}=6\cdot 24$, and that for every $S_1\neq S_2\in \mathcal{X}_1$ and every $P_1\neq P_2\in \mathcal{P}$ we have:
\begin{itemize}
\item $\abs{\Pi_0\cap S_1}=\abs{\Pi_0\cap P_1}=24,$
\item $\abs{\Pi_0\cap S_1\cap S_2}=\abs{\Pi_0\cap P_1\cap P_2}=0,$ 
\item $\abs{\Pi_0\cap S_1\cap P_1}=4.$
\end{itemize}
(the second equation is trivial for the $1$-point 
stabilizers, and it holds for the members of $\mathcal{P}$ too, since there is an outer 
involutory automorphism of $S_6$ that interchanges $1$-point stabilizers with 
the members of $\mathcal{P}$). \\
By applying an inclusion/exclusion argument there are only two ways to 
cover $\Pi_0$ with no more than six of these proper subgroups, either using all of
the six elements of $\mathcal{X}_1$, or all of the elements of $\mathcal{P}$. 
It follows that $\mathcal{D}$ is either $\mathcal{X}_1$ or $\mathcal{P}$. 
In both cases we have a contradiction, since $\mathcal{X}_1$ does not cover 
the $2$-elements of type $\Pi_1=(2,2,2)$, while $\mathcal{P}$ does not cover 
the $2$-cycles. We proved therefore that $A_6\in\mathcal{D}$.

We set $\mathcal{D}_1=\mathcal{D}\setminus\set{A_6}$. 
The collection $\mathcal{D}_1$ consists of at most five subgroups, which 
should cover the set of odd $2$-elements, that is the set 
$\Pi_1\cup\Pi_2\cup\Pi_3$, where $\Pi_1=(2,2,2)$, $\Pi_2=(4,1,1)$ and $\Pi_3=(2,1,1,1,1)$. The following table shows
the sizes of the intersections of these classes with the maximal subgroups (different form $A_6$).

\begin{center}
\noindent
\begin{tabular}{|c|c|c|c|c|c|c|}
\hline
$\Pi_i$ & $\abs{\Pi_i}$ & $\abs{\Pi_i\cap\mathcal{X}_1}$ 
  & $\abs{\Pi_i\cap\mathcal{X}_2}$ & $\abs{\Pi_i\cap\mathcal{W}_3}$
  & $\abs{\Pi_i\cap\mathcal{W}_2}$  & $\abs{\Pi_i\cap\mathcal{P}}$\\
\hline
$(2,2,2)$     & $15$ & $0$ & $3$  & $6$ & $7$ & $10$ \\
$(4,1,1)$     & $90$ & $30$ & $6$ & $0$ & $6$ & $30$ \\  
$(2,1,1,1,1)$ & $15$ & $10$ & $7$ & $6$ & $3$ & $0$  \\
\hline
\end{tabular}
\end{center}

We claim that in order to cover the class $\Pi_2$ of $4$-cycles we need to take 
either at least three different elements of $\mathcal{X}_1$ or at least three 
different elements of $\mathcal{P}$. This comes from the fact that, 
for every $S_1\neq S_2\in \mathcal{X}_1$ and every 
$P_1\neq P_2\in \mathcal{P}$, the following holds
\begin{itemize}
\item $\abs{\Pi_2\cap S_1\cap S_2}=\abs{\Pi_2\cap P_1\cap P_2}=6,$
\item $\abs{\Pi_2\cap S_1\cap P_1}= 10,$
\item $\abs{\Pi_2\cap S_1\cap S_2\cap P_1}=\abs{\Pi_2\cap S_1\cap P_1\cap P_2}=2,$
\item $\abs{\Pi_2\cap S_1\cap S_2\cap P_1\cap P_2}\leq 2,$
\end{itemize}
hence
$$\abs{\Pi_2\cap \left(S_1\cup S_2\cup P_1\cup P_2\right)}\leq 76.$$ 
Assume that $\mathcal{D}_1$ contains three different elements of 
$\mathcal{P}$, then, by looking at  the last line of the Table, the class
$\Pi_3$ should be covered using just two different subgroups, say 
$A,B\in\mathcal{X}_1\cup\mathcal{X}_2\cup\mathcal{W}_3$. this is impossible, 
since:
\begin{itemize}
\item if $A,B\in \mathcal{X}_1$, then $\abs{\Pi_3\cap A\cap B}=6$, 
\item if $A\in \mathcal{X}_1$ and $B\in \mathcal{X}_2$, then $\abs{\Pi_3\cap A
\cap B}\geq 4$,
\item if $A\in \mathcal{X}_1$ and $B\in \mathcal{W}_3$, then 
$\abs{\Pi_3\cap A\cap B}= 4$.
\end{itemize}
The opposite case when $\mathcal{D}_1$ contains three different 
elements of $\mathcal{X}_1$, follows immediately by using the duality of the 
outer automorphism of order two of $S_6$ (or with similar arguments applied to 
the first line of the table). 
\end{proof}

\begin{lem} \label{lem:subsum}
Let $a_1,\ldots,a_r,a$ be integers with $0 \leq a_1 \leq \ldots \leq a_r$, $r \geq 2$ and $a \geq 1$.
\begin{enumerate}
    \item If $\sum_{i=1}^r 2^{a_i} = 2^a$ then there exists a subset $J \subseteq \{1,\ldots,r\}$ such that $\sum_{i \in J} 2^{a_i} = 2^{a-1}$.
    \item If $\sum_{i=1}^r 2^{a_i} = 3 \cdot 2^a$ then one of the following occurs.
    \begin{itemize}
        \item $r=2$, $a_1=a$, $a_2=a+1$.
        \item $r=3$, $a_1=a_2=a_3=a$.
        \item There exist two disjoint subsets $J_1,J_2 \subseteq \{1,\ldots,r\}$ such that $$\sum_{i \in J_1} 2^{a_i} = \sum_{i \in J_2} 2^{a_i} = 2^{a-1}.$$
    \end{itemize}
\end{enumerate}
\end{lem}

\begin{proof}
Item (1). Assume that what we want to prove is false, and let $r$ be minimal with the property that $\sum_{i=1}^r 2^{a_i} = 2^a$ and there is no $J \subseteq \{1,\ldots,r\}$ with $\sum_{i \in J} 2^{a_i} = 2^{a-1}$. This implies $r \geq 3$. Note that since twice a power of $2$ is a power of $2$, we have $a_1 < \ldots < a_r \leq a-2$, hence 
$$2^a = \sum_{i=1}^r 2^{a_i} \leq \sum_{i=0}^{a-2} 2^i \leq 2^{a-1}-1,$$ 
a contradiction.

Item (2). Assume that $\sum_{i=1}^r 2^{a_i}=3 \cdot 2^a$. By item (1), to conclude it is enough to show that there exists $J \subseteq \{1,\ldots,r\}$ such that $\sum_{i \in J} 2^{a_i} = 2^a$, so suppose this is not the case, by contradiction. As for item (1), we may assume that $a_1 < \ldots < a_r$. We know that $a_r \leq a+1$ and $a_{r-1} \neq a$, therefore 
$$3 \cdot 2^a = \sum_{i=1}^r 2^{a_i} \leq \sum_{i=1}^{r-1} 2^{a_i} + 2^{a+1} \leq \sum_{i=0}^{a-1} 2^i + 2^{a+1} \leq 2^a-1+2^{a+1},$$ 
a contradiction.
\end{proof}

\begin{pro}\label{pro:main 32^a}
Let $n= 2^{a+1}+2^{a}=3\cdot 2^a$, with $a\geq 2$.
Then 
$$c_1\leq \sigma_0(S_n)\leq c_2,$$
where 
\begin{align*}
c_1 &=\begin{cases} 117 & \textrm{ if } a=2,\\
                 1+{n-1\choose 2^{a}-1} 
                 & \textrm{ if } a\geq 3,
      \end{cases}\\ 
c_2& = 2+{n-1\choose 2^{a}-1}+
\sum_{i=2}^{2^{a+1}}{n-i\choose 2^{a-1}-1}
\end{align*}
\end{pro}

\begin{proof}
To prove the upper bound we consider the collection
$$\mathcal{C}_2=\bigcup_{i=0}^{m}\mathcal{M}_i,$$
where $m=2^{a+1}+1$ and
\begin{align*}
\mathcal{M}_0&=\set{A_n},\\
\mathcal{M}_1&=\set{\Stab_{S_n}(U)\vert 1\in U, \, \abs{U}=2^a},\\
\mathcal{M}_i&=\set{\Stab_{S_n}(V)\vert i\in V\subset\set{i,\ldots, n},\, 
\abs{V}= 2^{a-1}}\!,\, 
\textrm{for } i=2,\ldots, 2^{a+1},\\
\mathcal{M}_{m}&=\set{\Stab_{S_n}(\set{m, \ldots ,n})}.
\end{align*}
The primary elements of odd order as well as the $2$-elements that are even
permutations are covered by $A_n$. 
Let $g$ be an odd $2$-element. If $g\in (2^a,2^a,2^a)$ then $g$ is covered by 
a unique subgroup in $\mathcal{M}_1$. Otherwise, since $g$ is an odd permutation, by Lemma \ref{lem:subsum} there are at least two disjoint subsets of cardinality $2^{a-1}$ on which $g$ acts. If one of these two contains the point $1$, then again $g$ is covered by 
a unique element of $\mathcal{M}_1$, 
otherwise $g$ lies in the stabilizer of a subset of size $2^{a-1}$ and 
not containing 1.
In this case $g$ is covered by a subgroup that lies in 
$\bigcup_{i=2}^m\mathcal{M}_i$. Since $\abs{\mathcal{C}_2}=c_2$ the 
upper bound is proved.\\

We prove now that $c_1$ is a lower bound for $\sigma_0(S_n)$.\\
Assume first that $a=2$, that is $n=12$. Arguing in a similar way as the cases $n=5,6$ or $10$, it is straightforward to 
prove that the alternating group $A_{12}$ belongs to every minimal primary 
covering. Now, the maximal subgroups that contain most elements of type
$\Pi_1=(4,4,4)$ are the ones in $\mathcal{W}_6$, each of which, by Lemma \ref{lem:S_n_impr} contains exactly $8\abs{\Pi_1}(6!)^2/12!$ such
elements. Since $12!/8(6!)^2=115,5$ we conclude that a minimal primary
covering has at least $117$ elements.

Let $a\geq 3$. We prove that $c_1$ is a lower bound for $\sigma_0(S_n)$ by 
showing that the collection $\mathcal{M}_1$ is definitely unbeatable on 
$\Pi_1=(2^a,2^a,2^a)$. 

Conditions (1) and (2) of Definition \ref{def:def_unbeat}
follow immediately. Let us prove condition (3). Note that the only maximal 
subgroups $M$ having nontrivial intersection with $\Pi_1$ are 
either stabilizers of a set of cardinality $2^a$, or imprimitive, or 
proper primitive maximal subgroups, that is 
elements of $\mathcal{W}\cup\mathcal{P}$. For such $M$ we define 
$c(M):=|M \cap \Pi_1|/|M_1 \cap \Pi_1|$ for $M_1 \in \mathcal{M}_1$, and we will 
prove that $c(M) \leq 1$. \\
In the first case we have that 
$$\abs{M\cap \Pi_1}=\abs{M_1\cap \Pi_1}=\frac{(2^a!)(2^{a+1}!)}{2^{3a+1}},$$
for every $M_1\in\mathcal{M}_1$ and every $M=\Stab_{S_n}(U)$, $1\not\in U$, $
\abs{U}=2^a$. Therefore $c(M)=1$. \\
Assume $M\in\mathcal{W}_{n/2}$. Then by Lemma \ref{lem:S_n_impr},
$$\abs{M\cap \Pi_1}=\frac{4\abs{\Pi_1}\abs{W_{n/2}}}{n!}=\frac{4}{3\cdot 2^{3a}}
\left((n/2)!\right)^2,$$
and therefore
\begin{align*}
c(M)&=\frac{8}{3}\cdot \frac{((n/2)!)^2}{(2^{a+1}!)(2^a!)}=\frac{8}{3}\cdot 
\frac{((2^a+2^{a-1})!)^2}{(2^{a+1})!(2^a)!}\\
    &=\frac{8}{3}\cdot \frac{(3t!)^2}{(4t)!(2t)!}=\frac{8}{3}\cdot \frac{{6t
    \choose 2t}}{{6t\choose 3t}},
\end{align*}
where $t=2^{a-1}$. Therefore $c(M)< 1$ for every $t \geq 4$, that is for 
every $a\geq 3$.  

Assume $M\in\mathcal{W}_{d}$, with $d\neq n/2$. Then 
$\abs{M\cap \Pi_1}< \abs{M}\leq \abs{W_{n/3}}$ by Lemma 
\ref{lem:order_max_impr}, therefore 
\begin{align*}
c(M)& < \frac{\abs{W_{n/3}}}{\abs{M_1\cap \Pi_1}}=\frac{(n/3!)^3\cdot 6\cdot 
2^{3a+1}}{(2^a!)(2^{a+1}!)}\\
&< \frac{2^{3a+4}}{{2^{a+1}\choose 2^a}}\leq \frac{2^{3a+4}(2^{a+1}+1)}{2^{2^{a
+1}}}<\frac{2^{4a+6}}{2^{2^{a+1}}}\leq 1,
\end{align*}
for every $a\geq 4$. The case $a=3$ can be checked directly. 

Finally assume that $M\in\mathcal{P}$. If $a\geq 3$, then $\abs{M}\leq 2^n$ 
by \cite{Maroti2002}, hence
$$c(M)< \frac{\abs{M}}{\abs{M_1\cap \Pi_1}}<
\frac{2^{n+3a+1}}{(2^{a+1}!)(2^a!)}<1.$$ 
The proof is completed.
\end{proof}

\textbf{Acknowledgments.} The authors are grateful to the referee for 
completing the calculations of $\sigma_0(S_{10})$ using  \cite{gap,gurobi}.


\begin{thebibliography}{20}

\bibitem{Cohn} Cohn, J. H. E., On $n$-sum groups, Math. Scand. Volume 75, Number 1, Year 1944, Pages 44--58.

\bibitem{DM} Dixon, John D., Mortimer, Brian, Permutation groups. Graduate Texts in Mathematics, 163. Springer-Verlag, New York, 1996.

\bibitem{Kantor} Fein, B. and Kantor, W. M. and Schacher, M., Relative {B}rauer groups. {II}, J. Reine Angew. Math., Volume 328, Year 1981, Pages 39--57.

\bibitem{Fuma} Fumagalli, F., On the indices of maximal subgroups and the normal primary coverings of finite groups. J. of Group Theory, Volume 22, Number 6, Year 2019, Pages 1015--1034.

\bibitem{Gasch} Gasch\"{u}tz, W., Die {E}ulersche {F}unktion endlicher aufl\"{o}sbarer {G}ruppen. Illinois J. Math., Volume 3, Year 1959, Pages 469--476.

\bibitem{gap}
The GAP~Group.
\newblock {\em GAP -- Groups, Algorithms, and Programming, Version 4.7.5},
  2014.

\bibitem{gurobi} Gurobi Optimizer Reference Manual, Gurobi Optimization, Inc., 2014, (http://www.gurobi.com).

\bibitem{Hup} Huppert, B., Endliche {G}ruppen. {I}, Die Grundlehren der Mathematischen Wissenschaften, Band 134, Year 1967, Pages xii+793.

\bibitem{Isaacs}
Isaacs, I. Martin, Finite Group Theory. Graduate Studies in Mathematics, 92. American Mathematical Society, Providence, RI, 2008. xii+350 pp.

\bibitem{Jones} Jones, G. A., Cyclic regular subgroups of primitive permutation groups, J. Group Theory, Volume 5, Year 2002, Number 4, Pages 403--407.

\bibitem{Kappe} Kappe, Luise-Charlotte and Nikolova-Popova, Daniela and Swartz, Eric, On the covering number of small symmetric groups and some sporadic groups, Groups Complex. Cryptol., Volume 8, Year 2016, Number 2, Pages 135--154.

\bibitem{Maroti2005} Mar\'{o}ti, A., Covering the symmetric groups with proper subgroups, J. Combin. Theory Ser. A, Volume 110, Year 2005, Number 1, Pages 97--111.

\bibitem{Maroti2002} Mar\'{o}ti, A., On the orders of primitive groups, J. Algebra, Volume 258, Year 2002, Number 2, Pages 631--640.

\bibitem{Opp} Oppenheim, Ryan and Swartz, Eric, On the covering number of {$S_{14}$}, Involve, Volume 12, Year 2019, Number 1, Pages 89--96, ISSN 1944-4176.

\bibitem{Robinson} Robinson, Derek J.S., A course in the theory of groups. Second 
edition. Graduate Texts in Mathematics, 80. Springer-Verlag, New York, 1996.


\bibitem{Swartz} Swartz, E., On the covering number of symmetric groups having degree divisible by six, Discrete Mathematics, Volume 339, Year 2016, Number 11, Pages 2593--2604.

\bibitem{Tom} Tomkinson, M. J., Groups as the union of proper subgroups, Math. Scand., Volume 81, Year 1997, Number 2, Pages 191--198.
	
\end{thebibliography}
\end{document}